\documentclass[a4paper, 11pt]{article}
\usepackage[dvipdfmx]{hyperref} 
\usepackage{srcltx}
\usepackage{indentfirst}
\usepackage{graphicx}
\usepackage{overpic}
\usepackage{multirow}
\usepackage[dvipdfmx]{hyperref}
\usepackage[hang]{subfigure}
\usepackage{amsmath, amssymb, amsthm}
\usepackage{color}
\usepackage[mathscr]{euscript}
\usepackage{mathrsfs} 

\newtheorem{theorem}{Theorem}
\newtheorem{lemma}{Lemma}

\newtheorem{corollary}{Corollary}

\newcommand\bbR{\mathbb{R}}
\newcommand\bbZ{\mathbb{Z}}
\newcommand\bbN{\mathbb{N}}

\newcommand\be{\boldsymbol{e}}

\newcommand\bB{{\bf{B}}}

\newcommand\dd{\,\mathrm{d}}
\newcommand\de{\mathrm{e}}

\newcommand \Vw{\mathcal{V}}  
\newcommand \VwR{\tilde{\mathcal{V}}} 
\newcommand \Thth{\mathcal{T}^h} 
\newcommand \bfThth{\mathbf{\it T}^h} 
\newcommand \Rv{ \mathbb{R}_v } 
\newcommand \Rh{ R^h } 
\newcommand \ThetaV{ {\Theta[V]} } 

\newcommand \FTFF[3]{\mathcal{F}_{ {#1}\rightarrow{#2}  } \left( #3 \right) }

\newcommand \IFTFF[3]{\mathcal{F}^{-1}_{ {#1}\rightarrow{#2}  } \left( #3 \right) }
\newcommand \zetah { {\zeta_h} } 

\newcommand \LLSpaceRv { {L^1 \left( (0,l); L^2(\Rv) \right)} } 
\newcommand \LLSpaceRy { {L^1 \left( (0,l); L^2(\bbR_y) \right)} }

\newcommand\pd[2]{\dfrac{\partial {#1}}{\partial {#2}}}
\newcommand\opd[2]{\dfrac{\dd {#1}}{\dd {#2}}}

\newcommand \ri{\mathrm{i}}

\newcommand\sinc{ {\rm sinc}\,}
\newcommand\supp{ {\rm supp}\,}

\newcommand \bz{ \mathbf{z} } 
\newcommand \bA{ \mathbf{A} } %
\newcommand \bh{ \mathbf{h} } %
\newcommand \br{ \mathbf{r} } %
\newcommand \bff{ {\mathbf{e}^h} }

\numberwithin{equation}{section}
\numberwithin{figure}{section}

\graphicspath{{images/}}

{\theoremstyle{remark} }

\title{Convergence of Semi-discrete Stationary Wigner Equation with
  Inflow Boundary Conditions}

\author{Ruo Li\thanks{HEDPS \& CAPT, LMAM \& School of Mathematical
    Sciences, Peking University, Beijing, China, email: {\tt
      rli@math.pku.edu.cn}.}, ~~ Tiao Lu\thanks{HEDPS \& CAPT, LMAM \&
    School of Mathematical Sciences, Peking University, Beijing,
    China, email: {\tt tlu@math.pku.edu.cn}.},
    ~~ Zhangpeng Sun \thanks{School of Mathematical Sciences, 
	  Peking University, Beijing, China, 
	  email: {\tt sunzhangpeng@pku.edu.cn}.}
}

\begin{document}
\maketitle

\begin{abstract}
  Making use of the Whittaker-Shannon interpolation formula with
  shifted sampling points, we propose in this paper a well-posed
  semi-discretization of the stationary Wigner equation with inflow
  BCs. The convergence of the solutions of the discrete
  problem to the continuous problem is then analysed, providing
  certain regularity of the solution of the continuous problem.

  \vspace*{4mm}


  \vspace*{4mm}

\end{abstract}

\section{Introduction} \label{sec:intro} 

The Wigner equation is one of the quantum frameworks equivalent to the
Schr\"odinger equation in some sense. The Wigner function is a
quasi-probability distribution introduced by Wigner in 1932 to study
quantum corrections to classical statistical mechanics
\cite{Wigner1932}. A great many applications of the Wigner equation
arose in pervasive fields, including statistical mechanics, quantum
optics, quantum chemistry, etc. Particularly, in the simulation of
nano-scale semiconductor devices, the Wigner equation was regarded as
a promising tool since it is the counterpart of the Boltzmann equation
in quantum mechanics. In 1987, Frensley \cite{Frensley1987}
numerically solved the stationary Wigner equation with inflow boundary
condition and successfully reproduced the negative differential
resistance phenomenon, which is a typical quantum effect verified by
experiments. This work motivated a lot of later work on numerical
simulations based on the Wigner equation \cite{Jensen1990,
  Tsuchiya1991, Shih1994, GeKo05, Querlioz2006, ShLu09, LiLuSun2014}. 
In these work, different boundary conditions are proposed for the 
stationary Wigner equation, e.g., absorbing boundary conditions 
\cite{Arnold1994} and device adaptive inflow boundary conditions
\cite{JiangLuCai2013}. Among these boundary conditions, the inflow
boundary condition is the most popular one due to its simplicity.

In spite of its popularity, the Wigner equation with inflow
boundary conditions (BCs) is far from thoroughly studied from a 
mathematical point of view. Numbers of mathematicians were then 
attracted to the study of the Wigner equation with inflow BCs, while
there are seldom results on the well-posedness of the problem yet. For
the time-dependent Wigner equation with inflow BCs,
well-posedness has been studied for the linear case
\cite{Markowich1989} and the nonlinear case
\cite{MarkowichDegond1990}, respectively. To the authors knowledge,
only one study has been carried out on the stationary Wigner equation,
where a rather involved technical method was used to construct a
solution \cite{Lange1997}. The stationary problem is even interesting
since it is applied to the current-voltage curve computation of
semiconductor devices in nano-scale, while a rigid proof of the unique
solvability has not yet been given. On the other hand, there are
comparatively fruitful studies in the numerical approximation
aspect. The well-posedness of the semi-discrete stationary Wigner
equation with inflow BCs has been proved in
\cite{ALZ00} if the velocity interval centered at zero is
neglected. The numerical convergence for the initial value problem has
been studied for the transient Wigner equation \cite{Ringhofer1992,
  Arnold1996, goudon02}. Using the Whittaker-Shannon interpolation
formula, Goudon in \cite{goudon02} constructed a converge sequence,
which are the solution of a semi-discrete version of the Wigner
equation, to approximate the solution of the transient Wigner
equation, in case that there exists a unique smooth solution of the
continuous problem.

Motivated by the work in \cite{ALZ00} and \cite{goudon02}, we consider
in this paper the convergence of the semi-discrete solution of the
stationary Wigner equation with the inflow BCs. The
Whittaker-Shannon interpolation formula in \cite{goudon02} is not able
to be applied to stationary problem, since it results in a singular
semi-discrete problem. We introduce a shift of the sampling points in
the Whittaker-Shannon interpolation formula thus the zero velocity is
excluded from the sampling points. Thus, the technique in \cite{ALZ00}
is applicable to prove the well-posedness of the semi-discrete problem
we propose. The well-posedness of the semi-discrete equation makes us
able to analyze the convergence of the solutions of the semi-discrete
problem to the continuous problem. It is proved that the convergence
rate depends only on the data and the regularity of the solution
of the continuous problem. As a necessary condition for any numerical
method, the well-posedness of the continuous problem definitely has to
be assumed, which provides us a solution with certain regularity thus
the numerical approximation is possible.

The rest part of this paper is arranged as follows. In Section 2, we
give the semi-discretization of the stationary Wigner equation with
inflow BCs based on Whittaker-Shannon interpolation
formula using shifted sampling points. In Section 3, we give an
estimate to the semi-discrete residual of the discretization as a
preparation to the final convergence result. In Section 4, the
convergence of the solution of the semi-discrete problem to the
continuous problem is clarified.  The Whittaker-Shannon
interpolation formula with shifted sampling points is collected in the
appendix for reference.

\section{Discretization}
We are considering the stationary dimensionless Wigner equation
\cite{Wigner1932}
\begin{equation}\label{eq:Wigner}
  v \pd{f(x,v)}{x}  - \ThetaV(f) = 0, 
\end{equation}
where the pseudo-differential operator $\ThetaV$ is defined by 
\begin{equation} \label{DefThetaV}
  \ThetaV(f)(x,v) = \ri \IFTFF{y}{v}{ D_V(x,y) \FTFF{v}{y}{f(x,v)} } , 
\end{equation} 
where $D_V(x,y)=V(x+y/2)-V(x-y/2)$ and $V(x)$ is the potential. The
Fourier transform of $u(v)$ and its inverse are standard as
\[
  \widehat{u}(y) =  \FTFF{v}{y}{u(v)} = \int_{\bbR} u(v) e^{-\ri v y } \dd y , 
\]
and
\[
  u(v) =  \IFTFF{y}{v}{\widehat{u}(y)} = \frac{1}{2\pi} \int_{\bbR}
  \widehat{u}(y) 
  e^{\ri v y } \dd y .  
\]
According to the convolution theorem of the Fourier transform, the
pseudo-differential operator defined in \eqref{DefThetaV} can be
written into   
\[
  \ThetaV(f) = \Vw(x,\cdot)\ast f(x,\cdot) 
  = \int_{\bbR} \Vw(x,v-v')f(x,v')\dd v' ,  
\]
where the Wigner potential $\Vw(x,v)$ is related to the potential
$V(x)$ through
\begin{equation} \label{VwDef}
  \Vw(x,v) = \ri \IFTFF{y}{v}{ D_V(x,y)} = \frac{\ri}{2\pi} \int_{-\infty}^{\infty}  
  D_V(x,y) \de^{\ri v y }\dd y.
\end{equation}
Then the Wigner equation is reformulated as
\begin{equation} \label{eq:cont}
  v \pd{f(x,v)}{x} - \int_{v' \in \bbR} \Vw(x,v-v') f(x,v') \dd v' =
  0,  \quad  (x,v) \in (0,l)\times \mathbb{R},  
\end{equation}
subject to the inflow boundary condition
\begin{equation} \label{eq:contBC}
  f(0,v) = f_b(v), \text{ if } v > 0, \quad 
  f(l,v) = f_b(v), \text{ if } v <0.  
\end{equation} 

We apply the Fourier transform to the Wigner equation \eqref{eq:cont}
and obtain
\begin{equation}\label{eq:WignerFT}
 \ri \pd{^2}{x \partial y} \widehat{f}(x,y)  - \ri D_V(x,y) \widehat{f}(x,y) = 0.  
\end{equation}
We introduce a smooth cutoff function $\zetah(y) \in
C_0^{\infty}(\bbR)$ as in \cite{Hormander1990} satisfying
\begin{equation} \label{zetaR}
  \zetah(y) = \zeta(y/\Rh),  \quad \Rh = \frac{1}{2h}  
\end{equation}
and
\[
  \left\{
  \begin{array}{l}
	 0 \leqslant  \zeta(y) \leqslant 1,  \\ 
	\zeta\left( y \right) =  \begin{cases}
	1, &\text{ on } B(0,1/2), \\
	0, &\text{ on } \bbR \backslash B(0,3/4).  
  \end{cases}
\end{array}
\right.
\]
And it is easy to see that 
\begin{equation} \label{zetahProp1} \zetah'(y)=0, \quad \text{if} \
  \ |y| \notin \left[ \frac{\Rh}{2}, \frac{3 \Rh}{4} \right].
\end{equation}
Furthermore, the derivative of the cutoff function $\zetah(y)$
\cite{Hormander1990} may satisfy
\begin{equation} \label{zetahProp2}
  |\zetah'(y)| \leq C_{\zeta} h ,  \quad \forall y \in \bbR ,   
\end{equation} 
where $C_{\zeta}$ is a constant independent of $h$. Multiplying
$\zetah(y)$ on both sides of \eqref{eq:WignerFT} yields
\begin{equation}
\zetah(y) \pd{^2}{x \partial y} \widehat{f}(x,y) 
- \zetah(y) D_V(x,y) \widehat{f}(x,y) = 0 . 
\end{equation}
Thus we have 
\begin{equation} \label{eq:WignerFTcutoff}
 \pd{^2}{x \partial y} ( \widehat{f}(x,y) \zetah(y) )  
 - D_V(x,y) \widehat{f}(x,y) \zetah(y) =
  \pd{}{x} \left( \widehat{f}(x,y) \zetah'(y)  \right).   
\end{equation}
Let
\[
  t^h(x,v)=\IFTFF{y}{v}{\widehat{f}(x,y)\zetah(y)},
\]
which is a function with a compactly supported Fourier transform,
precisely $\supp (\widehat{t^h}(x,y)) \subset \supp(\zetah(y)) \subset
B(0,\frac{3}{4}\Rh)$. According to the Shannon sampling theory,
$t^h(x,v)$ can be completely represented by the Whittaker-Shannon
interpolation formula \eqref{eq:Whittaker}
\begin{equation} \label{thxvSinc}
  t^h(x,v) = \sum_{n \in \bbZ} t^h(x,v_n) \sinc \left( \Rh 
	(v-v_n) 
  \right), 
\end{equation}
where $v_n =(n+1/2)\frac{\pi}{\Rh}$.  

We then apply the inverse Fourier transform to \eqref{eq:WignerFTcutoff} to
yield the equation of $t^h(x,v)$
\[
  v\pd{}{x} t^h(x,v) - \Theta[V] t^h(x,v) = \Thth(x,v),
\]
where 
\[
  \widehat{\Thth}(x,y) = \pd{}{x}\widehat{f}(x,y) \zetah'(y). 
\]
By Lemma \ref{lemma1}, we have that $\Theta[V] t^h(x,v) = \Theta[V
\chi_{B(0,\Rh)}] t^h(x,v)$ due to \eqref{convolution2}, thus
\begin{equation} \label{Eqthv}
  v\pd{}{x} t^h(x,v) - \Theta[V \chi_{B(0,\Rh)}] t^h(x,v) 
= \Thth(x,v).
\end{equation} 
By setting $v=v_n$ in \eqref{Eqthv}, we have
\begin{equation} \label{Eqthvn}
  v_n \opd{t^h(x,v_n)}{x} - (\Theta[V\chi_{B(0,\Rh)}] t^h)(x,v_n) 
= \Thth(x,v_n), \quad n \in \bbN.   
\end{equation} 
The Shannon sampling theorem tells one that \eqref{Eqthv} is
equivalent to the discrete-velocity equations \eqref{Eqthvn}. We point
out that
\[
(\Theta[V\chi_{B(0,\Rh)}] t^h)(x,v_n) = \frac{\pi}{\Rh}\sum_{m \in
  \bbZ} \VwR_{n-m} t^h(x, v_m),
\]
where $\VwR_n(x)$ is defined by
\[
\VwR_n(x) = \frac{\ri}{2\pi} \int_{B(0,\Rh)} D_V(x,y) e^{\ri y
  n\frac{\pi}{\Rh}} \dd y,
\]
and it is equal to the inverse Fourier transform of the truncated function
$\ri D_V(x,y)\chi_{B(0,R^h)}$ at velocity $\tilde{v}_n =n\pi/\Rh$
\begin{equation} \label{VwRDef}
  \VwR_n(x) = \ri \IFTFF{y}{v}{ D_V(x,y) \chi_{B(0,R^h)}(y)} 
  ( \tilde{v}_n ).  
\end{equation}
This allows us to reformulate \eqref{Eqthvn} as
\[
v_n \opd{t^h(x,v_n)}{x} - \frac{\pi}{\Rh}\sum_{m \in \bbZ} \VwR_{n-m}
t^h(x, v_m) = \Thth(x,v_n), \quad n \in \bbN.
\]

A reasonable problem one may be interested in is the case that
$\Thth(x,v)$ goes to zero as $h\rightarrow 0$. As a special setup, if
$\widehat{f}(x,y)$ has a compact support in $B(0,\Rh/2)$, then $\Thth(x,v)
= 0$. Hence we are motivated to propose the semi-discrete version of
the Wigner equation as
\begin{equation} \label{eq:discreteWigner}
  v_n \opd{f_n(x)}{x} - \frac{\pi}{\Rh}\sum_{m \in \bbZ} 
  \VwR_{n-m} f_m(x) = 0,
\end{equation}
subject to 
\begin{equation}\label{eq:discreteBC}
  f_n(0) = t^h_b(v_n), \text{ if } n \geqslant 0, \quad  
  f_n(l) = t^h_b(v_n), \text{ if } n < 0,  
\end{equation}
where 
\[
  t^h_b(v) = \FTFF{y}{v}{\IFTFF{v}{y}{f_b(v)} \zetah(y)}. 
\]
This is formulated as a boundary value problem (BVP). Since $v = 0$ is
excluded from the sampling points $v_n$, the method to prove the
well-posedness of the semi-discrete Wigner equation with inflow
boundary conditions in \cite{ALZ00} is then applicable to the BVP
\eqref{eq:discreteWigner}-\eqref{eq:discreteBC}. Here we directly
conclude that the BVP \eqref{eq:discreteWigner}-\eqref{eq:discreteBC}
admits a unique solution $f_n(x)$.

We let
\begin{equation}\label{eq:fhxL}
 f^h(x,v)=\sum_n f_n(x) \sinc\left( \Rh(v-v_n) \right),
\end{equation}
as the approximation of $t^h(x,v)$. If there is a fast enough decay of
$\widehat{f}(x,y)$ in terms of $y$, the residual term $\Thth(x,v)$ can be
arbitrary small as $h$ going to zero. With a small enough residual
$\Thth(x,v)$, not only the difference of $f^h(x,v)$ from $t^h(x,v)$ is
small, but also the difference between $t^h(x,v)$ and $f(x,v)$ may be
small. Consequently, it is expected that $f^h(x,v)$ is an appropriate
approximation of the continuous problem if there is a fast enough
decay of $\hat{f}(x,y)$ in terms of $y$. The major object in the rest
of this paper is to give the precise senses of this conclusion and its
rigid proof.

\section{Estimate of Semi-discrete Residual}\label{sec:non}

We denote the semi-discrete residual to be $e^h_n(x) = t^h(x,v_n) -
f_n(x)$. Comparing \eqref{Eqthvn} and \eqref{eq:discreteWigner}, we
have the equation for $e^h_n(x)$
\begin{equation} \label{eq:ehn}
  v_n \opd{e^h_n(x)}{x} - \frac{\pi}{\Rh} \sum_{m} \VwR(x,v_n-v_m) e^h_m(x) = \Thth(x,v_n). 
\end{equation} 
Clearly we have $e_n^h(0) = 0$ for $n \geq 0$ and $e_n^h(l) = 0$ for
$n < 0$ since the inflow BCs of $f^h_n(x)$ and
$t^h(x,v_n)$ are the same. This is again a BVP, while it is
nonhomogeneous. We directly extend the method in \cite{ALZ00} to this
nonhomogeneous BVP to give an upper bound estimate, which is used to
prove the convergence of the approximate solution.

At first, let us introduce the notations used in \cite{ALZ00}. From
the discrete equation \eqref{eq:ehn} of $e^h_n$ , we introduce vector
functions $ \mathbf{e}^h=\{e^h_{n}\}_{n\in Z}$, $ \bfThth =
\{\Thth_n\}_{n\in Z}$, then we have
\begin{equation} \label{eq:ehDiscrete}
 \mathbf{T} \frac{\dd \mathbf{e}^h(x)}{\dd x} - \mathbf{A}(x) \mathbf{e}^h
 =\bfThth
\end{equation}
with the BCs 
\begin{equation} \label{eq:ehBC}
\left\{
 \begin{aligned}
  \mathbf{e}^h_n(0)=0, & n\geqslant 0,\\
  \mathbf{e}^h_n(l)=0, & n<0 , 
 \end{aligned}
 \right. 
\end{equation}
where $\mathbf{T}$ and 
$\bA(x)$ are defined as
\begin{equation} \label{TandA}
  \mathbf{T}=\text{diag}(v_n)_{n\in\bbZ},
\quad \mathbf{A}(x)=\left(\frac{\pi}{\Rh}\VwR(x,v_n-v_m)\right)_{n,m
\in\bbZ}. 
\end{equation} 

We show below that for $0 \leqslant x \leqslant l$, $\mathbf{A}(x)$ is
a bounded linear operator on $H := l^2$ and $x \rightarrow
\mathbf{A}(x)$ is continuous in the uniform operator topology. Here
$l^2$ is the real Hilbert space with natural inner product $ \left(
  \mathbf{x}, \mathbf{y} \right) = \displaystyle \sum_{n \in \bbZ} x_n
y_n$. Notice that $\bA(x)$ is a representation based on sampling points
of $\Theta[V\chi_B(0,\Rh)]$ on $X^h = \{ e^h(x,v) \in L^2(\Rv):
\widehat{e}^h(x,y) \subset B(0,\Rh)\}$. According to Shannon sampling
theory, $\|e^h(x,v)\|_{L^2(\Rv)}^2 = \frac{\pi}{\Rh}
\|\be^h\|^2_{l^2}$. So $\be^h \in l^2 $ implies $e^h(x,v) \in X^h$.
By Lemma \ref{lemma1}, we can conclude that $\bA(x)$ can be defined as
\[
  (\bA(x)\be^h)_n  = (\Theta[V\chi_{B(0,\Rh)}]e^h)(x,v_n).  
\]
Thus we have 
\[
  \frac{\pi}{\Rh} \| \bA(x)\be^h \|_{l^2}^2  = 
  \|(\Theta[V\chi_{B(0,\Rh)}]e^h)(x,\cdot)\|_{L^2(\Rv)}^2.
\]
According to Parseval's theorem of the Fourier transform, we have
\[
  \frac{\pi}{\Rh} \| \bA(x)\be^h \|_{l^2}^2  = 
  \frac{1}{2\pi} \| D_V(x,y)\chi_{B(0,\Rh)} \hat{e}^h(x,y)\|^2_{L^2(\bbR_y)} 
  \leqslant  4\|V\|_{L^\infty}^2 \frac{\pi}{\Rh}
  \|\be^h\|^2_{l^2} . 
\]
Thus the norm of $\bA(x)$ is uniformly bounded by
\begin{equation} \label{eq:Abound}
  \|\bA(x) \| \leqslant 2 \|V\|_{L^\infty},
\end{equation}
and $\bA \in L^1((0,l); B(H))$, where $B(\cdot)$ is the space of linear
operator on a Hilbert space.

Following \cite{ALZ00}, we need to transform it into an initial value
problem (IVP) using the technique therein. At first, we
denote $\bbZ^- = \{n \in \bbZ: n < 0\}$ and $\bbZ^+ = \{n \in \bbZ:
n\geqslant 0\}$.  $H$ may be decomposed as $H = H^-\bigoplus H^+$
where $H^{\pm} = l^2(\bbZ^{\pm})$.  We denote by $Q^{\pm}$ the
restrictions of $H$ onto $H^{\pm}$, i.e., $Q^{\pm}\bff = \bff^{\pm}$
for any $\bff = (\bff^+,\bff^-)$, $\bff^{\pm}\in H^{\pm}$. Let
$P^{\pm}$ be the projections defined by $P^{+}\bff = (0, \bff^+)$,
$P^{-}\bff = (\bff^-,0)$, and the embeddings $E^{\pm}:
H^{\pm}\rightarrow H$ are defined by $E^{+}\bff^+ = (0,\bff^+)$,
$E^-\bff^-=(\bff^-,0)$. One has the relations that $P^{\pm} =
E^{\pm}Q^{\pm}$.

Since $\bA$ is clearly skew-symmetric, it is decomposed as 
\begin{equation} \label{eq:AStar}
  \bA(x) = \begin{pmatrix}
	\bA^{--} & \bA^{-+} \\
	\bA^{+-} & \bA^{++} 
  \end{pmatrix}
  = -\bA^{\ast} (x) 
\end{equation}
with $\bA^{++} = Q^+\bA E^{+} \in B(H^+)$, $\bA^{+-} = Q^+\bA E^{-}
\in B(H^-,H^+)$, $\bA^{-+} = Q^-\bA E^{+} \in B(H^+,H^-)$, $\bA^{--} =
Q^-\bA E^{-} \in B(H^-)$.  Also, one has
\begin{equation}
  D = \begin{pmatrix}
	D^{-} & 0 \\
  0 & D^+
  \end{pmatrix}
  , \quad 
  |D| = \begin{pmatrix}
	-D^{-} & 0 \\
  0 & D^+
  \end{pmatrix}
\end{equation}
where $D^{\pm} = {\rm diag}(1/v_n)_{n \in \bbZ^{\pm}}$. 
We get $|D|\geq 0$ in the Hilbert space sense, i.e., 
$\left <|D|\bff,\bff\right> \geq 0$ for every $\bff \in H$. 

Let $\bff = \sqrt{|D|} \bz $, and $\bz \in H$ implies $\bff \in
H$. Then the equation for $\bz$ is
\begin{equation} \label{eq:bz}
  \bz_x - \bB(x) \bz = \br,   \quad 0<x<l,  
\end{equation}
\begin{equation} \label{eq:bzBC}
  \bz^{+}(0) =0, \quad \bz^{-}(l) =0,  
\end{equation}
where $\br=(\br^+,\br^-) = {\sqrt{|D|}}^{-1} D \Thth$, the matrix
$\bB(x)$ is defined as $\bB(x) = \sqrt{|D|}^{-1} D \bA(x) \sqrt{|D|}$.
$\bA \in L^1( (0,l); B(H) )$ implies $\bB \in L^1( (0,l); B(H) )$
since $\sqrt{|D|} \in B(H)$. We may write $\bB(x)$ in the form
\begin{equation}
  \bB(x) = \begin{pmatrix}
	-\sqrt{-D^-} \bA^{--} (x) \sqrt{-D^-} & 
	-\sqrt{-D^-} \bA^{-+} (x) \sqrt{D^+} \\ 
	\sqrt{D^+} \bA^{+-} (x) \sqrt{-D^-}  &  
	\sqrt{D^+} \bA^{++} (x) \sqrt{D^+}   
  \end{pmatrix} . 
\end{equation}
By the norm of $\bA(x)$, it is clear that
\begin{equation} \label{eq:Bbound}
  \| \bB(x) \| \leq \dfrac{2}{\pi h } \| V \|_{L^\infty}.
\end{equation}

Lemma 3.1 in \cite{ALZ00} gave us the well-posedness for the
homogeneous BVP
\begin{equation} \label{homoeq}
  \bz_x - \bB(x) \bz = 0,   \quad 0<x<l,  
\end{equation}
\begin{equation} \label{homoIC}
  \bz(0) = \bz_0 \in H,
\end{equation}
as below:
\begin{lemma} [Lemma 3.1 in \cite{ALZ00}]
  Since $\bB \in L^1( (0,l); B(H) )$, the IVP \eqref{homoeq} -
  \eqref{homoIC} has a unique mild solution $\bz \in W^{1,1}(
  (0,l); H)$, and there exists a unique strongly continuous propagator
  $U(x,x') \in B(H)$, $\forall 0\leqslant x,x'\leqslant l$. It
  satisfies
  \begin{equation}
    \begin{array}{ll}
      \opd{U(x,0)}{x} - B(x) U(x,0) = 0,  & 
      \opd{U(0,x)}{x} + U(0,x)B(x) =0,   
    \end{array}
  \end{equation}
  almost everywhere on $(0,l)$.
\end{lemma}

The propagator $U$ in this lemma allows us to reformulate the BVP
\eqref{eq:bz} to an IVP. Acturally, the solution of the BVP
\eqref{eq:bz} satisfies
\begin{equation} \label{eq:ghEq}
  \bz(x) = U(x,0)
  \begin{pmatrix}
    \bh^-\\ 0  
  \end{pmatrix}
  + \int_{0}^{x} U(x,s) \br(s) \dd s 
  = U(x,l) 
  \begin{pmatrix}
    0 \\ \bh^+
  \end{pmatrix}
  + \int_{l}^{x} U(x,s) \br(s) \dd s,
\end{equation}
where $\bz^-(0)=\bh^-$, $\bz^+(l) =\bh^+$ are the corresponding
outflow data. The idea is to calculate $\bh^+$ from \eqref{eq:ghEq} by
eliminating $\bh^-$. Noting that $\bz(0) = (0,\bh^-)$ and $\bz(l) =
(\bh^+,0)$, we have the equations for $\bh^-$ and $\bh^+$
\begin{equation} \label{eq:314}
  \begin{pmatrix} 0 \\ \bh^+ \end{pmatrix}
= U(l,0)\begin{pmatrix} \bh^- \\ 0 \end{pmatrix}
	+ \int_0^l U(l,s)\br(s) \dd s,
\end{equation}
\begin{equation} \label{eq:315}
  \begin{pmatrix} \bh^- \\ 0 \end{pmatrix}
= U(0,l)\begin{pmatrix} 0 \\ \bh^+ \end{pmatrix}
	+ \int_l^0 U(0,s)\br(s) \dd s.
\end{equation}
Applying $P^+$ and $P^-$ on \eqref{eq:314} and 
\eqref{eq:315} respectively yields 
\begin{equation} \label{eq:316}
  \begin{pmatrix} 0 \\ \bh^+ \end{pmatrix}
  = P^{+} U(l,0)\begin{pmatrix} \bh^- \\ 0 \end{pmatrix}
	+ P^+ \int_0^l U(l,s)\br(s) \dd s, 
\end{equation}
\begin{equation} \label{eq:317}
  \begin{pmatrix} \bh^- \\ 0 \end{pmatrix}
= P^- U(0,l)\begin{pmatrix} 0 \\ \bh^+ \end{pmatrix}
	+ P^- \int_l^0 U(0,s)\br(s) \dd s.  
\end{equation}
Eliminating $\bh^-$ in \eqref{eq:317} and \eqref{eq:316}, we obtain
the equation for $\bh^+$ as
\begin{equation} \label{eq:417}
   (I-K) \begin{pmatrix} 0 \\ \bh^+  \end{pmatrix}
	  = P^+  U(l,0) P^-  \int_{l}^{0} U(0,s) \br(s) \dd s 
  + P^+ \int_{0}^{l} U(l,s)\br(s) \dd s,
\end{equation}
where 
\[
  K = P^+U(l,0)P^-U(0,l)P^+.
\]
Here the operator $K$ is the same as $K$ defined in \cite{ALZ00} (page
7173 Eq. (3.17)) for the homogeneous case. Making use of the
skew-symmetry of $\bA(x)$, it is proved in \cite{ALZ00} that $K$ is
negative, thus $I-K$ is invertible with a bounded inverse. We are then
instantly inferred that
\[
\lVert (I-K)^{-1} \rVert \le 1.
\]
As a result, it is concluded that the nonhomogeneous BVP can be
transfomred into an IVP, as the extension of Theorem 3.3 in
\cite{ALZ00}. Precisely, we have the following lemma:
\begin{lemma}\label{lemma:non}
  The nonhomogeneous BVP \eqref{eq:bz} - \eqref{eq:bzBC} has a unique
  mild solution $\bz \in W^{1,1}( (0,l); H)$ and
  \[
	\|\bz(x)\|_{l^2} \leqslant 3 \exp \left( \dfrac{6 l \| V
	\|_{L^\infty}}{ \pi h } \right) \int_0^l \|\br(s)\|_{l^2}
  \dd s.
  \]
\end{lemma}
\begin{proof}
  Given by \cite{ALZ00}, the self-adjointness of the bounded operator
  $K$ imply that $I-K$ is invertible with a bounded inverse, which
  shows the unique solvability of the BVP \eqref{eq:bz} -
  \eqref{eq:bzBC}. In the following, we are going to estimate $\lVert
  \bz (x) \rVert_{l^2}$.

  In the chapter 5 of \cite{Pazy1992}, it shows for every $ 0\leq x
  \leq x' \leq l$, $U(x, x')$ is a bounded linear operator and
  \[
    \| U(x, x')\| \le \exp \left( \int_x^{x'} \|\bB(s)\| \dd s \right),
  \]
  thus due to \eqref{eq:Bbound},
  \begin{equation} \label{eq:esti-U}
	\|U(x,x')\| \leqslant \exp \left( \dfrac{ 2 |x' - x| \| V
	\|_{L^\infty}}{\pi h } \right) \leqslant \exp 
	\left( \dfrac{ 2 l \| V \|_{L^\infty}}{  \pi h} \right).
  \end{equation}
  By \eqref{eq:ghEq}, we have
  \begin{equation} \label{eq:esti-bz}
    \begin{aligned}
	  \lVert \bz(x) \rVert_{l^2} & \leqslant \lVert U(x,l) \rVert ~ \lVert
      \bh^+ \rVert_{l^2} + \int_0^l \lVert U(x,s) \rVert ~ \lVert
      \br(s) \rVert_{l^2} \dd s \\
      & \leqslant \lVert U(x,l) \rVert ~ \lVert \bh^+ \rVert_{l^2} +
	  \exp \left( \dfrac{ 2 l \| V \|_{L^\infty}}{ \pi h}
	  \right) \int_0^l \|\br(s)\|_{l^2} \dd s \\
	  & \leqslant \exp \left( \dfrac{2 l \| V \|_{L^\infty}}{
  \pi h } \right) \left( \lVert \bh^+ \rVert_{l^2} + \int_0^l
          \|\br(s)\|_{l^2} \dd s \right).
    \end{aligned}
  \end{equation}
  Since $\|(I-K)^{-1}\|\le 1$ , $\|P^+ \|\le 1 $, $\|P^- \|\le 1$ and
  by \eqref{eq:esti-U}, we estimate $\bh^+$ using \eqref{eq:417} to
  have
  \begin{equation} \label{eq:hpestimate}
    \lVert \bh^+ \rVert_{l^2} \leqslant \left( \exp \left( \dfrac{
		  2 l 
	  \| V \|_{L^\infty}}{ \pi h}  \right) + \exp \left( 
	\dfrac{4 l \| V \|_{L^\infty}}{ \pi h} \right) \right) 
    \int_0^l \lVert \br(s) \rVert_{l^2} \dd s. 
  \end{equation}
  Substituting \eqref{eq:hpestimate} into \eqref{eq:esti-bz} yields
  the estimate for $\| \bz(x) \|_{l^2}$, i.e., 
  \[
  \|\bz(x)\|_{l^2} \leqslant 3 \exp \left( \dfrac{6 l \| V
      \|_{L^\infty}}{ \pi h} \right) \int_0^l \|\br(s)\|_{l^2}
  \dd s.
  \]
 This ends the proof.
\end{proof}

Recalling the relation that $\bff=\sqrt{|D|} \bz$, we immediately
deduce the estimate for the original BVP \eqref{eq:ehDiscrete} -
\eqref{eq:ehBC} from Lemma \ref{lemma:non}. We remark that $\bz\in H$
if and only if $\bff \in \tilde{H}$ where the space
$\tilde{H}=l^2(\bbZ; |v_n|)$ is a weighted $l^2$-space endowed with
the inner product
\[
  (x,y)_{\tilde{H}}:=\sum_{j\in J} |v_j|x_j y_j.
\]
\begin{corollary} \label{lemma:5} The BVP \eqref{eq:ehDiscrete} -
  \eqref{eq:ehBC} has a unique mild solution $\be^h \in W^{1,1}(
  (0,l); \tilde{H})$, $\mathbf{T} \opd{\be^h}{x} \in L^1((0,l); H)$ and
  \begin{equation}
    \|\be^h \|_{\tilde{H}}
    \leqslant
	\frac{3}{\sqrt{\pi h}} \exp \left( \dfrac{6 l 
	  \|V\|_{L^\infty}}{ \pi h} 
	\right) \int_0^l \| \bfThth(s)\|_{l^2} \dd s ,
  \end{equation}
\end{corollary}
\begin{proof}
  Noticing that $\|\br(x)\|_{l^2} \leqslant \frac{1}{\sqrt{\pi h}}
  \|\br(x)\|_{\tilde{H}} =\frac{1}{\sqrt{\pi h}} \|\bfThth(x)\|_{l^2}$,
  the result is inferred by Lemma \ref{lemma:non}.
\end{proof}

\section{Convergence}
By Corollary \ref{lemma:5} and the triangle inequality 
\begin{equation} \label{eq:inequality}
    \|f^h(x,v) - f(x,v)\|_{L^2(\Rv)} \leqslant 
    \|f^h(x,v) - t^h(x,v)\|_{L^2(\Rv)} 
	+ \|f(x,v) - t^h(x,v)\|_{L^2(\Rv)},
\end{equation}
the term $\|f(x,v) - t^h(x,v)\|_{L^2(\Rv)}$ has to be estimated to
have the final result on $\| f^h(x,v) -f(x,v) \|_{L^2(\Rv)}$.
Obviously, $\|f(x,v) - t^h(x,v)\|_{L^2(\Rv)}$ is not going to zero as
$h \rightarrow 0$ without any assumption on $f(x,v)$. Let us assume
that $f(x,v)$ satisfies
\[
  f(x,v) \in C\left( [0,l]; L^2\left( \Rv \right) \right)
  \cap C^{1}\left( (0,l); L^2\left(\Rv \right) \right).   
\]
Though this is not a rigour constraint on $f(x,v)$, it is enough to
provide us the corresponding convergence. Since $t^h(x,v)$ is
approximating $f(x,v)$ using the Whittaker-Shannon interpolation
formula, which is a spectral expansion, a successful approximation to
$f(x,v)$ has to require a certain decay in Fourier space. With the
enhanced assumption that the Fourier transformation of $f(x,v)$ is
decaying exponentially, a spectral convergence may be
achieved. Precisely, from the fact that the compactly supported smooth
functions $C_c^{\infty}(\bbR)$ are dense in $L^2(\bbR)$ and the fact
that Fourier transform is a unitary transform on $L^2(\bbR)$, the
estimate of $\| t^h-f \|_{\LLSpaceRv}$ is given in the following
lemma.
\begin{lemma} \label{lemmaAppr1} Let $f(x,v) \in \LLSpaceRv$
  and $t^h(x,v) = \IFTFF{y}{v}{\widehat{f}(x,y)\zetah(y)}$ where
  $\zetah$ is defined in \eqref{zetaR}, then
  \[
  \lim_{h \rightarrow 0^+} \| f - t^h \|_{\LLSpaceRv} = 0.
  \]
  Furthermore, if there exists a constant $\alpha >0$ such that
  $\widehat{f}(x,y) \exp (\alpha |y|) \in \LLSpaceRy $, then $f^h$
  converges to $f$ with an exponential rate
  \[
  \| f - t^h \|_{\LLSpaceRv}  \leqslant  C \exp \left(-\frac{\alpha }{4h} \right), 
  \]
  where $C = \frac{1}{\sqrt{2\pi}}\| \widehat{f}(x,y) \exp(\alpha|y|)\|_{\LLSpaceRy}$
  does not dependent on $h$.  
\end{lemma}
\begin{proof}
  By the Parseval theorem of the Fourier transform, we have  
  \begin{equation} \label{es1}
    \begin{split}
	\|f - t^h\|_{\LLSpaceRv}  &= 
	\frac{1}{\sqrt{2\pi}} \| \widehat{f} - \widehat{t^h} \|_{\LLSpaceRy}  \\
	&= \frac{1}{\sqrt{2\pi}}
	\int_0^{l} \| \widehat{f}(x,\cdot) - \widehat{f}(x,\cdot)\zetah(\cdot) \|
    _{L^2(\bbR_y)} \dd x. 
    \end{split}
  \end{equation}
  According to the definition of $\zetah(y)$ in \eqref{zetaR}, we have 
  \begin{equation}
	\int_0^{l} \| \widehat{f}(x,\cdot) - \widehat{f}(x,\cdot)\zetah(\cdot) \|
    _{L^2(\bbR_y)} \dd x 
	\leqslant 
	\int_0^l \left(  \int_{|y| \in 
	[\frac{\Rh}{2},\frac{3\Rh}{4}] }
  | \widehat{f}(x,y) |^2 \dd y \right)^{1/2} \dd x.
  \end{equation}
  It is clear the right hand side is going to zero as $h \rightarrow
  0$. If $\widehat{f}(x,y) \exp (\alpha |y|) \in \LLSpaceRy $, obviously
  we have
  \[
  \int_0^{l} \| \widehat{f}(x,\cdot) - \widehat{f}(x,\cdot)\zetah(\cdot) \|
  _{L^2(\bbR_y)} \dd x 
  \leqslant 
  \exp \left(-\frac{\alpha \Rh}{2} \right)
  \int_0^l \left(  \int_{\bbR_y}
    \left| \widehat{f}(x,y) \exp (\alpha |y|) \right|^2 \dd y \right)^{1/2} \dd x.
  \]
  Noticing that $\Rh = \dfrac{1}{2h}$, we finish the proof.  
\end{proof}
In order to the estimate of $\|e^h\|_{\LLSpaceRv}$ using Corollary 1,
we give the estimate $ \int_{0}^{l} \|\bfThth(x)\|_{l^2} \dd x = 
\sqrt{2\pi h}  \int_{0}^{l}\|\Thth(x,\cdot)\|_{L^2(\Rv)}\dd x 
$ in the following lemma.  
\begin{lemma}\label{lemmaAppr2}
  If there exists a constant $\alpha>0$ such that
  $\pd{\widehat{f}(x,y)}{x} \exp (\alpha |y| ) \in \LLSpaceRy $, then
  \begin{equation} \label{eq:RESbound}
    \|\Thth\|_{\LLSpaceRv} \leqslant C 
    h \exp \left(-\frac{\alpha}{4h} \right),
  \end{equation}
  where $C =
  \dfrac{C_{\zeta}}{\sqrt{2\pi}} \left\|
  \pd{\widehat{f}(x,y)}{x} \exp(\alpha|y|) \right\|_{\LLSpaceRy} $.
\end{lemma}

\begin{proof}
By the Parseval theorem of the Fourier transform, we have 
  \[
  \lVert \Thth \rVert_{L^2(\Rv)} ^2 =
  \frac{1}{2\pi}\lVert \widehat{\Thth}  \rVert_{L^2(\bbR_y)} ^2  
  =\frac{1}{2\pi} \left\| \pd{}{x}\widehat{f}(x,y) \zetah'(y) 
  \right\|_{L^2(\bbR_y)} ^2.  
\]
Using the properties \eqref{zetahProp1} and \eqref{zetahProp2} of the
cutoff function, we obtain
  \[
  \left\| \pd{}{x}\widehat{f}(x,y) \zetah'(y) 
  \right\|_{L^2(\bbR_y)} ^2  \leqslant 
  C_{\zeta}^2 h^2 \exp \left(-\frac{\alpha \Rh}{2}\right) \int
  _{\bbR_y }  \left \vert
  \pd{}{x}\widehat f(x,y) \exp(\alpha|y|) \right \vert ^2   \dd y .
  \]
Thus, we have 
\begin{equation}
    \|\Thth\|_{\LLSpaceRv} \leqslant 
	\frac{C_{\zeta} h}{\sqrt{2\pi}} 
	\exp \left(-\frac{\alpha}{4h} \right) \left\|
	\pd{\widehat{f}(x,y)}{x} \right\|_{\LLSpaceRy}.
\end{equation} 
  This gives us \eqref{eq:RESbound}.
\end{proof}

We are now ready to give the major result:
\begin{theorem}
  Let $V(x) \in L^{\infty}(\bbR)$. 
  If the continuous BVP \eqref{eq:cont}-\eqref{eq:contBC} has a unique
  solution $f(x,v) \in C^0( [0,l];L^2(\Rv)) \cap
	C^1( (0,l);L^2(\Rv)) $, and there exists a
	constant $\alpha > \dfrac{24 l}{\pi}\|V\|_{L^{\infty}}$ 
  such that $\widehat{f}(x,y) \exp (\alpha |y|)
  \in W^{1,1}( (0,l); L^2(\bbR_y))$, then 
\[
  \|f^h - f\|_{\LLSpaceRv} 
  \leqslant  C \exp \left(-\frac{\beta}{h} \right),
\]
where $C  = \max\left( \dfrac{1}{\sqrt{2\pi}}, \dfrac{ 3 C_{\zeta} }{
\sqrt{2}\pi^{3/2} } \right)
\left\|\widehat{f}(x,y) \exp (\alpha |y|) \right\|_{
W^{1,1}\left( (0,l); L^2(\bbR_y)   \right)} $
and $\beta =  \dfrac{\alpha}{4} -\dfrac{6 l}
{\pi}\|V\|_{L^{\infty}}$. 
\end{theorem}
\begin{proof}
  By Lemma \ref{lemmaAppr1}, we have
  \begin{equation} \label{the1}
    \| f - t^h \| _{\LLSpaceRv} \leqslant C_1 \exp
	\left(-\frac{\alpha}{4h }\right),  
  \end{equation}
  where $C = \frac{1}{\sqrt{2\pi}} \|
  \widehat{f}(x,y) \exp(\alpha|y|) \|_{\LLSpaceRy}$. 

Using the facts that
 \[
   \|f^h(x,v)-t^h(x)\|_{L^2(\Rv)} = \|e^h(x,v)\|_{L^2(\Rv)} =
  \sqrt{2\pi h}  \|\be^h\|_{l^2},
\]
\[
   \|\Thth(x,\cdot)\|_{L^2(\Rv)} =
   \sqrt{2\pi h} \|\bfThth(x)\|_{l^2} , 
\] 
\[
  \| \be^h\|_{l^2} \leq \frac{1}{\sqrt{\pi h}} \| \be^h\|_{\tilde{H}}, 
\]
and by Lemma \ref{lemmaAppr2} and Corollary \ref{lemma:5}, we have
  \begin{equation} \label{the2}
    \| f^h(x,v)-t^h(x,v)\|_{\LLSpaceRv} \leqslant 
	C_2
	\exp \left( - \left( \frac{\alpha}{4} -
	  \frac{6l}{\pi}\|V\|_{L^{\infty}} \right)\frac{1}{h} 
	\right),  
  \end{equation}
  where $C_2= \dfrac{3C_{\zeta}}{\sqrt{2}\pi^{3/2}} \left\|
    \pd{\widehat{f}(x,y)}{x} \exp (\alpha |y|) \right\|_{\LLSpaceRy}
  $.  Then we finish the proof by \eqref{the1}, \eqref{the2} and the
  triangle inequality \eqref{eq:inequality}.
\end{proof}

\section*{Acknowledgements}
This research was supported in part by the National Basic Research
Program of China (2011CB309704) and NSFC (91230107, 11325102,
91330205).

\begin{appendix}
\section{Shannon Sampling Theory}
We consider $f(v)$ to be a smooth function of $v$ in the sense its 
Fourier transform has a compact support, i.e., 
${\rm supp}(\hat{f}) \subset B(0,\Rh)$. According to the Shannon 
sampling theory (a lot of references, e.g., \cite{Marks1991}), the 
function $f(v)$ can be represented by  
\begin{equation} \label{eq:Whittaker}
  f(v) = \sum_{n=-\infty}^{\infty} f\left( v_n \right)
  \sinc \left( \Rh\left( v-v_n \right) \right),
\end{equation}
where $h = \frac{1}{2\Rh}$  $v_n=(n+1/2)\frac{\pi}{\Rh}=
2\pi  (n+1/2) h $ is the $n$-th sampling point. 
The sampling frequency is higher than twice of the highest
frequency of $f$, since $\supp{\hat{f}}\subset B(0,\Rh)$ implies 
the largest $|y|$ satisfying $\hat{f}(y)\neq 0$ is smaller than $\Rh$.  
Thus $f(v)$ can be completely reconstructed by its values at the
sampling points through \eqref{eq:Whittaker}.
The $\sinc$ function is defined by
\begin{equation*}
  \sinc x = \frac{\sin x}{x}. 
\end{equation*}
\eqref{eq:Whittaker} is then called the Whittaker-Shannon interpolation
formula. 

From the convolution theorem of the Fourier transform, it is easy to
know that if $f$ has a compact supported Fourier transform,
i.e. $\supp(\hat{f}) \subset B(0,\Rh)$, then their convolution 
$f\ast g(v) = \int_{\bbR} f(v-v') g(v') \dd v' $
has a compact supported Fourier transform 
\[
\FTFF{v}{y}{f\ast g (v)} = \hat{f}\hat{g}   
\]
with 
\[
\supp\left( \FTFF{v}{y}{f\ast g(v)} \right) \subset B(0,\Rh). 
\]
 Thus  $(f\ast g)(v)$ can be 
represented by the Whittaker-Shannon interpolation formula. 
Explicitly, we have the following lemma to represent $(f\ast g)(v)$. 
\begin{lemma} \label{lemma1}
Let $f(v)$ be a function with a compactly supported
Fourier transform  satisfying $\supp (\hat{f}) \subset B(0,\Rh) $.   
Let $g(v) \in L^2(\bbR)$. Let $L>2\Rh$, $h=\frac{1}{2\Rh}$,   
$v_n =(n+1/2)2\pi h$ and $\tilde{v}_n= n 2\pi h$.   
Then the convolution of $f$ and $g$ can be 
expressed with the Whittaker-Shannon interpolation formula
\eqref{eq:Whittaker}, 
\begin{equation}\label{convolution2}
f\ast g(v) = \frac{2\pi}{L}\sum_{n\in\bbZ}
\sum_{m\in\bbZ} g_{n-m}f_m 
\sinc \left(\frac{L}{2}(v-v_n)\right),  
\end{equation} 
where 
\begin{equation}
f_n=f(v_n), \quad  g_n = \left(\mathcal{F}^{-1}_{y\rightarrow
  v}\left(\hat{g}(y)\chi_{B(0,\Rh)}\right)\right)(\tilde{v}_n).
\end{equation}
\end{lemma}
\begin{proof}
By the convolution theorem of the Fourier transform, we have
\begin{equation}
  f\ast g(v) = \int_{-\infty}^{\infty} f(v-v')g(v') \dd v' 
  = \IFTFF{y}{v} {\hat{f}\hat{g}}. 
\end{equation}
Using $\supp (\hat{f}) \subset B(0,\Rh)$, we have   
\begin{equation} \label{212Con}
  f\ast g(v)  =\IFTFF{y}{v}{\hat{f}\hat{g}}
  =\IFTFF{y}{v} {\hat{f}\hat{g} \chi_{B(0,\Rh)}}
  = f \ast \tilde{g}
\end{equation}
where $\tilde{g}(v) = \IFTFF{y}{v}{\hat{g}(y)\chi_{B(0,\Rh)(y)}}$.   
Both $f$ and $\tilde{g}$ have a compactly supported Fourier transform 
contained in $B(0,\Rh)$ result in  
\begin{equation} \label{gPres}
  f(v) = \sum_{n} f_n\sinc\left( \Rh ( v-v_n \right) ,
  \quad 
  \tilde{g}(v) = \sum_{n} g_n\sinc\left( \Rh ( v-\tilde{v}_n \right) ,
\end{equation}
where $g_n = \tilde{g}(\tilde{v}_n)$. 
Plugging  \eqref{gPres} into \eqref{212Con} and making use  of 
using the following property (c.f. Page 13 of \cite{Marks1991})
\begin{equation} \label{SincProperty}
  \int_{-\infty}^{\infty}
  \sinc \left(\frac{L}{2} (v-\tilde{v}_n)\right) 
  \sinc \left( \frac{L}{2} (v-\tilde{v}_{m}) \right)
\dd v = \frac{2\pi}{L} \delta_{nm} ,
\end{equation}
yields \eqref{convolution2}.
 \end{proof}
\end{appendix}


\end{document}